\documentclass[11pt,reqno]{amsart}
\usepackage{hyperref}
\usepackage{amssymb,amsmath,amsfonts,latexsym}
\usepackage{bm}
\usepackage{mathtools}
\usepackage{enumerate}
\usepackage{array,graphics,color}
\allowdisplaybreaks

\numberwithin{equation}{section}
\newtheorem{dfn}{Definition}[section]
\newtheorem{thm}[dfn]{Theorem}
\newtheorem{lma}[dfn]{Lemma}

\newtheorem{crlre}[dfn]{Corollary}

\newtheorem{rmrk}[dfn]{Remark}

\DeclarePairedDelimiterX{\norm}[1]{\lVert}{\rVert}{#1}
\DeclarePairedDelimiterX{\bnorm}[1]{\big\lVert}{\big\rVert}{#1}
\DeclarePairedDelimiterX{\Bnorm}[1]{\Big\lVert}{\Big\rVert}{#1}
\newcommand\at[2]{\left.#1\right|_{#2}}

\newcommand{\R}{\mathbb{R}}

\newcommand{\hil}{\mathcal{H}}
\newcommand{\hils}{\mathcal{B}_2(\hil)}
\newcommand{\boh}{\mathcal{B}_1(\hil)}
\newcommand{\dds}{\dfrac{d}{ds}}
\newcommand{\ddt}{\dfrac{d}{dt}}
\newcommand{\ipi}{\int_{0}^{2\pi}}
\newcommand{\cir}{\mathbb{T}}

\newcommand{\la}{\langle}
\newcommand{\ra}{\rangle}
\newcommand{\tl}{\tau_{_l}}
\newcommand{\fl}{f_{_l}}

\newcommand{\gkl}{g_{_{kl}}}

\begin{document}

\title[Koplienko-Neidhardt trace formula for unitaries $-$ A new proof]{Koplienko-Neidhardt trace formula for unitaries $-$ A new proof}

\dedicatory{Dedicated to Professor Kalyan Bidhan Sinha on the occasion of his 75th birthday}

\author[Chattopadhyay] {Arup Chattopadhyay}
\address{Department of Mathematics, Indian Institute of Technology Guwahati, Guwahati, 781039, India}
\email{arupchatt@iitg.ac.in, 2003arupchattopadhyay@gmail.com}

\author[Das]{Soma Das}
\address{Department of Mathematics, Indian Institute of Technology Guwahati, Guwahati, 781039, India}
\email{soma18@iitg.ac.in, dsoma994@gmail.com}

\author[Pradhan]{Chandan Pradhan}
\address{Department of Mathematics, Indian Institute of Technology Guwahati, Guwahati, 781039, India}
\email{chandan.math@iitg.ac.in, chandan.pradhan2108@gmail.com}



\begin{abstract}
 Koplienko \cite{Ko} found a trace formula for perturbations of self-adjoint operators by operators of Hilbert-Schmidt class $\mathcal{B}_2(\mathcal{H})$. Later in 1988, a similar formula was obtained by Neidhardt \cite{NH} in the case of unitary operators. In this article, we give a still another proof of Koplienko-Neidhardt trace formula in the case of unitary operators by reducing the problem to a finite dimensional one as in the proof of Krein's trace formula by Voiculescu \cite{Voi}, Sinha and Mohapatra \cite{MoSi94,MoSi96}. 

\end{abstract}
\maketitle

\section{Introduction}
One of the fundamental concept in perturbation theory is the existence of spectral shift function and the associated trace formula. The notion of first order spectral shift function originated from Lifshits' work on theoretical physics \cite{Lif} and later the mathematical theory of this object elaborated by M.G. Krein in a series of papers, starting with \cite{Kr53}. In \cite{Kr53} (see also \cite{Kr83}), Krein proved that given two self-adjoint operators $H$ and $H_0$ (possibly unbounded) such that $H-H_0$ is trace class, then there exists a unique real valued $L^1(\mathbb{R})$ function $\xi$  such that 
\begin{equation}\label{inteq1}
 \text{Tr}~\{\phi(H)-\phi(H_0)\} = \int_{\mathbb{R}} \phi'(\lambda)~\xi(\lambda)~d\lambda
\end{equation}
holds for sufficiently nice functions $\phi$. The function $\xi$ is known as Krein's spectral shift function and the relation \eqref{inteq1} is called Krein's trace formula. The original proof of Krein uses analytic function theory. Later in \cite{BS2} (see also \cite{BS1}), Birman and Solomyak approached the trace formula \eqref{inteq1} using the theory of double operator integrals, though they failed to prove the absolute continuity of  the spectral shift. In 1985, Voiculescu \cite{Voi} gave an alternative proof of the trace formula \eqref{inteq1} by adapting the proof of classical Weyl-von Neumann theorem for the case of bounded self-adjoint operators and later Sinha and Mohapatra extended Voiculescu's method to the unbounded self-adjoint \cite{MoSi94} and unitary cases \cite{MoSi96}.  Recently, Peller \cite{Pe16} describe completely the class of functions (viz, the class of operator Lipschitz functions
on $\mathbb{R}$), for which the Krein’s trace formula \eqref{inteq1} holds. A similar result was obtained by Krein in \cite{Kr62} for pair of unitary operators $\big\{U,U_0\big\}$ such that $U-U_0$ is trace class. For each such pair there exists a real valued $L^1([0,2\pi])$- function $\xi$, unique modulo an additive constant, (called a spectral shift function for $\{U,U_0\}$) such that
 \begin{equation}\label{intequ2}
 \text{Tr}~\big\{\phi(U)-\phi(U_0)\big\} = \int_{0}^{2\pi} \frac{d}{dt}\big\{\phi(\textup{e}^{it})\big\}~\xi(t)~dt,
\end{equation}
whenever $\phi'$ has absolutely convergent Fourier series. Recently, Aleksandrov and Peller \cite{AP16} extended the formula \eqref{intequ2} for arbitrary operator Lipschitz functions $\phi$ on the unit circle $\mathbb{T}$.

The modified second order spectral shift function for Hilbert-Schmidt perturbations was introduced by Koplienko in \cite{Ko}. Let $H$ and $H_0$ be two self-adjoint operators in a separable Hilbert space $\mathcal{H}$ such that $H-H_0=V\in \mathcal{B}_2(\mathcal{H})$. In this case the difference $\phi(H)-\phi(H_0)$ is no longer of trace-class and one has to consider instead
\[
 \phi(H)-\phi(H_0)-\at{\dfrac{d}{ds}\Big(\phi(H_0+sV)\Big)}{s=0},
\]
where $\at{\dfrac{d}{ds}\Big(\phi(H_0+sV)\Big)}{s=0}$ denotes the G$\hat{a}$teaux derivative of $\phi$ at $H_0$ in the direction $V$ (see \cite{RB}) and and find a trace formula for the above expression under certain assumptions on $\phi$. Under the above hypothesis, Koplienko's formula asserts that there exists a unique function $\eta \in L^1(\mathbb{R})$ such that
\begin{equation}\label{intequ3}
 \operatorname{Tr}\Big\{\phi(H)-\phi(H_0)-\at{\dfrac{d}{ds}\Big(\phi(H_0+sV)\Big)}{s=0}\Big\}=\int_{\mathbb{R}} \phi''(\lambda)~\eta(\lambda)~d\lambda
\end{equation}
for rational functions $\phi$ with poles off $\mathbb{R}$. The function $\eta$ is known as Koplienko spectral shift function corresponding to the pair $(H_0,H)$. In 2007, Gesztesy et al. \cite{GePu} gave an alternative proof of the formula \eqref{intequ3} for the bounded case and in 2009, Dykema and Skripka \cite{DS,SK10}, and earlier Boyadzhiev \cite{BO} obtained the formula \eqref{intequ3} in the semi-finite von Neumann algebra setting. Later in 2012, Sinha and the first author of this article provide an alternative proof of the formula \eqref{intequ3} using
the idea of finite dimensional approximation method as in the works of Voiculescu \cite{Voi}, Sinha and Mohapatra \cite{MoSi94,MoSi96}. In this connection it is worth mentioning that in 1984, Koplienko also conjectured about the existence of the higher order spectral shift measures $\nu_n$, $n>2$, for the perturbation $V\in \mathcal{B}_n(\mathcal{H})$ and it is remarkable to note that recently Potapov, Skripka and Sukochev resolve affirmatively Koplienko's conjecture and establishes the existence of higher order spectral shift function in their outstanding and beautiful paper \cite{PoSkSu13In} using the concept of multiple operator integral.

A similar problem for unitary operators was considered by Neidhardt \cite{NH}. Let $U$ and $U_0$ be two unitary operators on a separable Hilbert space $\hil$ such that $U-U_0~\in~\mathcal{B}_2(\hil)$. Then  $U=e^{iA}U_0$, where $A$ is a self-adjoint operator in $\mathcal{B}_2(\hil)$. Denote $U_s=e^{isA}U_0,~s\in \mathbb{R}$. Then it was shown in \cite{NH} that there exists a $L^1([0,2\pi]))$-function $\eta$ (unique upto an additive constant ) such that 
\begin{equation}\label{intequ4}
 \operatorname{Tr}\Big\{\phi(U)-\phi(U_0)-\at{\dfrac{d}{ds}\phi(U_s)}{s=0}\Big\}=\int_{0}^{2\pi} \dfrac{d^2}{dt^2} \big\{\phi (e^{it})\big\} \eta(t) dt,
\end{equation}
whenever $\phi''$ has absolutely convergent Fourier series. The function $\eta$ is known as Koplienko spectral shift function corresponding to the pair $(U_0,U)$. In \cite{Pe05}, Peller obtained better sufficient conditions on functions $\phi$, under which trace formulae \eqref{intequ3} and \eqref{intequ4} hold. In this connection, it is also worth mentioning that recently Potapov, Skripka and Sukochev proved higher order analogs of the formula \eqref{intequ4} in \cite{PoSkSu16}. For more about trace formulas and related topics, we refer the reader to (\cite{ MNP17, MNP1717, MNP19, Peller85, Peller06,  PoSkSu13, PoSkSu, PoSu, Sk, ST}) and the references cited therein. 

In this article we once again supply the new proof of Koplienko-Neidhardt trace formula \eqref{intequ4}, we believe for the first time, using the idea of finite dimensional approximation method as in the works of Voiculescu, Sinha and Mohapatra, referred earlier. The rest of the paper is organized as follows: In Section 2, we give a proof of Koplienko-Neidhardt trace formula when $dim~\mathcal{H}<\infty$. Section 3 is devoted to the reduction of the problem to finite dimensions and in Section 4 we prove the trace formula by a limiting argument. 

\section{Koplienko-Neidhardt trace formula in finite dimension}
	Here, $\mathcal{H}$ will denote the separable Hilbert space we work in;  $\mathcal{B}(\mathcal{H})$, $\mathcal{B}_1(\mathcal{H})$, $\mathcal{B}_2(\mathcal{H})$ the set of bounded, trace class, Hilbert-Schmidt class operators in $\mathcal{H}$ respectively with $\|\cdot\|,\|\cdot\|_1, \|\cdot\|_2$ as the associated norms and $\text{Tr} \{A\}$ denote the trace of a trace class operator $A$.
	\begin{thm}\label{th1}
		Let $U$ and $U_0$ be two unitary operators on a separable Hilbert space $\hil$ such that $U-U_0~\in~\mathcal{B}_2(\hil)$. Then there exists a self-adjoint operator $A\in~\mathcal{B}_2(\hil)$ such that $U=e^{iA}U_0$. 
	\end{thm}
	\begin{proof}
		Since $UU_0^*$ is a unitary operator, then there is a self-adjoint operator $A$ with the spectrum in $(-\pi,\pi]$ (that is, $\sigma(A)\subseteq(-\pi,\pi]$) such that  $UU_0^*=e^{iA}$ and hence $U=e^{iA}U_0$. Let $\{f_i\}$ be any orthonormal basis of $\hil$. Then
		from the inequality $|x|\leq~\dfrac{\pi}{2}|e^{ix}-1|$ for $x\in{(-\pi,\pi]}$ and by using the spectral theorem  we conclude 
		\begin{align*}
		||A||_2^2 =\sum_{i=1}^{\infty}||Af_i||^2
		=\sum_{i=1}^{\infty}\int_{-\pi}^{\pi}|\lambda|^2~ ||E(d\lambda) f_i||^2
		& \leq~\dfrac{\pi^2}{4}\sum_{i=1}^{\infty}\int_{-\pi}^{\pi}|e^{i\lambda}-1|^2~||E(d\lambda) f_i||^2\\
		& =\dfrac{\pi^2}{4}\left\|e^{iA}-I\right\|_2^2 = \dfrac{\pi^2}{4}\|U-U_0\|_2^2,
		\end{align*}
		where $E(\cdot)$ is the spectral measure corresponding to the self-adjoint operator $A$. Thus from the hypothesis we conclude that $A\in\mathcal{B}_2(\hil)$. This completes the proof. 
	\end{proof}
	The following theorem states Koplienko-Neidhardt trace formula in finite dimension.  
	\begin{thm}\label{th2}
		Let $U$ and $U_0$ be two unitary operators in a separable Hilbert space $\hil$ such that $U-U_0~\in~\mathcal{B}_2(\hil)$ and $p(\lambda)=\lambda^r~(r\in\mathbb{Z}),~\lambda~\in~\cir$.
		\begin{enumerate}[(i)]
			\item Then 
			{\begin{align}\label{derivexpeq}
			&\frac{d}{ds}(p(U_s))=\begin{cases}
			\sum\limits_{k=0}^{r-1}U_s^{r-k-1}~(iA)~ U_s^{k+1} &\text{ if } r\geq 1,\\
			\qquad 0 &\text{ if } r=0,\\
			-\sum\limits_{k=0}^{|r|-1}(U_s^*)^{|r|-k}~(iA)~( U_s^*)^{k} &\text{ if } r\leq -1,
			\end{cases}
			\end{align}} where $U_s=e^{isA}U_0,~ s\in\R$.
			\item If furthermore $\dim(\hil)<\infty$, then there exists a $L^1([0,2\pi])$- function $\eta$ (unique upto an additive constant) such that
			\begin{equation}\label{11}
			\operatorname{Tr}\Big\{p(U)-p(U_0)-\at{\dfrac{d}{ds}(p(U_s))}{s=0}\Big\}=\int_{0}^{2\pi}\frac{d^2}{dt^2}\big\{p(e^{it})\big\}~\eta(t)~dt,
			\end{equation} where $p(\cdot)$ is any trigonometric polynomial on $\cir$ with complex coefficients and 
			\begin{equation}\label{eqfin0}
			\eta(t)=
			 \int_{0}^{1} \operatorname{Tr}\big\{A[E_0(t)-E_s(t)]\big\}ds~,~~t\in[0,2\pi]
			\end{equation}
 where $E_s(\cdot)$ is the spectral measure of the unitary operator $U_s$. Moreover, 
			\begin{equation}\label{eqfin1}
			 \operatorname{Tr}\Big\{p(U)-p(U_0)-\at{\dfrac{d}{ds}(p(U_s))}{s=0}\Big\}=\int_{0}^{2\pi}\frac{d^2}{dt^2}\big\{p(e^{it})\big\}~\eta_o(t)~dt,
			\end{equation}
where \begin{equation*}
       \eta_o(t) = \eta(t) -\frac{1}{2\pi} \int_0^{2\pi} \eta(s) ds~,~~t\in[0,2\pi]\quad \text{and}\quad \|\eta_o\|_{L^1([0,2\pi])}\leq \frac{\pi}{2} \|A\|_2^2. 
      \end{equation*}
     \end{enumerate}
	\end{thm}
	\begin{proof}
$(i)$~Since $U-U_0\in\mathcal{B}_2(\hil)$, then by the above Theorem ~\ref{th1} there exists a self-adjoint operator $A\in~\mathcal{B}_2(\hil)$ such that $U=e^{iA}U_0$. Denote $U_s=e^{isA}U_0,~s\in \R$ and note that each $U_s$ is an unitary operator. For $p(\lambda)=\lambda^r ~(r\geq 1),~\lambda\in \cir$, we have
\begin{align*}
\begin{split}
\frac{p(U_{s+h})-p(U_s)}{h} 
= \frac{1}{h}\sum_{k=0}^{r-1}
U_{s+h}^{r-k-1}\left[U_{s+h}-U_s\right]U_s^k = \frac{1}{h}\sum_{k=0}^{r-1}
U_{s+h}^{r-k-1} \left[e^{ihA}-I\right]U_s^{k+1}, 
\end{split}
\end{align*}
which converges in operator norm to 
$$ \sum_{k=0}^{r-1}U_s^{r-k-1}~(iA)~U_s^{k+1} \quad \text{as}\quad h\rightarrow 0.$$
Similarly for $p(\lambda)=\lambda^r ~(r\leq -1),~\lambda\in \cir$, we have
\begin{align*}
\frac{p(U_{s+h})-p(U_s)}{h} 
=&~\frac{1}{h}\sum_{k=0}^{|r|-1}
(U_{s+h}^*)^{|r|-k-1}\left[U_{s+h}^*-U_s^*\right](U_s^*)^k\\
 =&~\frac{1}{h}\sum_{k=0}^{|r|-1}
(U_{s+h}^*)^{|r|-k-1}(U_s^*)\left[e^{-ihA}-I\right](U_s^*)^{k}, 
\end{align*}
which again converges in operator norm to 
$$ -\sum_{k=0}^{r-1}(U_{s}^*)^{|r|-k}(iA)(U_s^*)^{k} \quad \text{as}\quad h\rightarrow 0.$$
$(ii)$~By using the cyclicity of trace and noting that the trace now is a finite
sum, we have that for $p(\lambda)=\lambda^r ~(r\geq 1),~\lambda\in \cir$,
{\begin{align*}
&\operatorname{Tr}\Big\{p(U)-p(U_0)-\at{\dds p(U_s)}{s=0}\Big\}
=\operatorname{Tr}\Big\{\int_{0}^{1} \frac{d}{ds}(p(U_s))~ds\Big\} - \operatorname{Tr}\Big\{ \at{\dds p(U_s)}{s=0}\Big\}\\
&=\int_{0}^{1}\operatorname{Tr}\Big\{ \sum_{k=0}^{r-1}~U_s^{r-k-1}~(iA)~U_s^{k+1}\Big\}~ds -\operatorname{Tr}\Big\{\sum_{k=0}^{r-1}~ U_0^{r-k-1}~(iA)~U_0^{k+1}~\Big\}\\
&=\int_{0}^{1} r~\operatorname{Tr}\Big(iA U_s^{r}\Big)~ds -\int_{0}^{1} r~\operatorname{Tr}\Big(iA U_0^r\Big)~ ds 
= \operatorname{Tr}\Big\{r (iA) \int_{0}^{1} ds~\int_{0}^{2\pi} e^{irt}\big(E_s(dt)-E_0(dt)\big)\Big\},
\end{align*}}
where $E_s(\cdot)$ and $E_0(\cdot)$ are the spectral measures determined uniquely by the unitary operators $U_s$ and $U_0$ respectively such that the spectral measures are continuous at $t=0$, that is, $E_s(0)=0=E_0(0)$ (see page 281, \cite{RN}). Next by performing integration by-parts we have that 
\begin{align*}
& \operatorname{Tr}\Big\{p(U)-p(U_0)-\at{\dds p(U_s)}{s=0}\Big\} \\
&= \operatorname{Tr}\Big\{ r (iA) \int_{0}^{1} ds~  \Big(e^{irt} \big[E_s(t)-E_0(t)\big]\Big|_{t=0}^{2\pi}-ir \int_{0}^{2\pi}e^{irt} \big[E_s(t)-E_0(t)\big] dt\Big)~\Big\}\\
&=\int_{0}^{2\pi}(ir)^2 e^{irt}
\left(\int_0^1 \operatorname{Tr}\big\{A[E_0(t)-E_s(t)\big]\big\} ds\right) = \int_{0}^{2\pi} \frac{d^2}{dt^2}[p(e^{it})]~\eta(t)~dt,
\end{align*} 
where we have set $$\eta(t)=\int_{0}^{1} \operatorname{Tr}\big\{A [E_0(t)-E_s(t)]\big\} ds.$$
In similar manner, we can prove the identity (\ref{11}) for $p(\lambda)=\lambda^r ~(r\leq -1),~\lambda\in \cir$.\\
Now it is clear that $\eta~\in~L^1([0,2\pi])$ and therefore it makes sense to define
$$\eta_o(t)=\eta(t)-\dfrac{1}{2\pi}\int_{0}^{2\pi}\eta(s)ds,~~~t\in[0,2\pi].$$
Thus the assertion \eqref{eqfin1} follows from the following observation
\begin{align*}
\int_{0}^{2\pi} e^{imt}\eta_0(t)&=\int_{0}^{2\pi} e^{imt}\left[\eta(t)-\dfrac{1}{2\pi}\int_{0}^{2\pi}\eta(s)ds\right]dt\\
&=\int_{0}^{2\pi} e^{imt}\eta(t)dt-\dfrac{1}{2\pi}\int_{0}^{2\pi} \eta(s)ds \int_{0}^{2\pi} e^{imt}dt =\int_{0}^{2\pi} e^{imt}\eta(t)dt  \quad \text{for} \quad m\in \mathbb{Z}\setminus \{0\}.
\end{align*}
Let $f\in L^{\infty}([0,2\pi])$, and consider
\begin{align*}
f_o=f-\dfrac{1}{2\pi}\int_{0}^{2\pi}f(s)ds.
\end{align*}
Then it is easy to observe that
\begin{align*}
\int_{0}^{2\pi} f(t) \eta_{o}(t) dt=\int_{0}^{2\pi} f_o(t) \eta(t) dt,\quad \int_{0}^{2\pi}f_o(t)dt=0, \quad \text{and} \quad \norm{f_o}_\infty\leq 2\norm{f}_\infty.
\end{align*}
Therefore by using the expression \eqref{eqfin0} of $\eta$ and using Fubini's theorem to interchange the orders of integration and integrating by-parts, we have for $g(e^{it})=\int_{0}^{t} f_0(s)ds, t\in [0,2\pi]$ that
\begin{align}\label{eqfin2}
\nonumber  & \int_{0}^{2\pi} f(t)\eta_o(t)dt = \int_{0}^{2\pi} f_o(t)\eta(t)dt
=\int_{0}^{2\pi} \frac{d}{dt}[g(e^{it})]\left(\int_{0}^{1} \operatorname{Tr}[A(E_0(t)-E_s(t))]ds\right)dt\\
\nonumber  &=\int_{0}^{1} ds\int_{0}^{2\pi} \frac{d}{dt}[g(e^{it})] \operatorname{Tr}[A(E_0(t)-E_s(t))]dt\\
\nonumber  &=\int_{0}^{1} ds\left\{g(e^{it})\operatorname{Tr}[A(E_0(t)-E_s(t))]\Big|_{~t=0}^{~2\pi}-\int_{0}^{2\pi} g(e^{it})\operatorname{Tr}[A(E_0(dt)-E_s(dt))]\right\}\\
&=-\int_{0}^{1} ds \int_{0}^{2\pi} g(e^{it})\operatorname{Tr}[A(E_0(dt)-E_s(dt))]
=\int_{0}^{1} \operatorname{Tr}[A\{g(U_s)-g(U_0)\}] ds.
\end{align}
On the other hand by using the idea of double operator integrals, introduced by Birman and Solomyak \cite{BS1,BS2,BS3} we have
\begin{align}\label{eqfin3}
 \nonumber g(U_s)-g(U_0) & = \ipi\ipi \left[g(e^{i\lambda})-g(e^{i\mu})\right] E_s(d\lambda)E_0(d\mu) \\
\nonumber  & =\ipi\ipi\dfrac{g(e^{i\lambda})-g(e^{i\mu})}{e^{i\lambda}-e^{i\mu}}~E_s(d\lambda)(U_s-U_0)E_0(d\mu)\\
& = \ipi\ipi\dfrac{g(e^{i\lambda})-g(e^{i\mu})}{e^{i\lambda}-e^{i\mu}} ~\mathcal{G}(d\lambda\times d\mu)(U_s-U_0),
\end{align}
where $\mathcal{G}(\Delta\times \delta )(V)= E_s(\Delta)VE_0(\delta)$ ($V\in \mathcal{B}_2(\hil)$ and $\Delta\times \delta \subseteq \mathbb{R}\times \mathbb{R}$) extends to a spectral measure on $\mathbb{R}^2$ in the Hilbert space $\mathcal{B}_2(\hil)$ (equipped with the inner product derived from the trace) and its total variation is less than or equal to
$\|V\|_2$. Thus by using the standard inequality $\vline\dfrac{g(e^{i\lambda})-g(e^{i\mu})}{e^{i\lambda}-e^{i\mu}}\vline \leq \dfrac{\pi}{2} \|f_o\|_\infty\leq \pi \|f\|_{\infty}$, for $\lambda,\mu\in [0,2\pi]$, we conclude from \eqref{eqfin3} that
\begin{align}\label{mainest1}
 \norm{g(U_s)-g(U_0)}_2\leq \pi \norm{f}_\infty\norm{U_s-U_0}_2,
\end{align}
which combining with \eqref{eqfin2} implies that
\begin{align*}
 &\vline \int_{0}^{2\pi} f(t)\eta_o(t)dt~\vline~\leq~\int_{0}^{1} ~||A||_2 ~||g(U_s)-g(U_0)||_2~ds 
 \leq~\pi \norm{f}_\infty\norm{A}_2\int_{0}^{1}~\norm{U_s-U_0}_2~ds\\ 
& \leq~\pi \norm{f}_\infty\norm{A}_2\int_{0}^{1}~s\norm{A}_2~ds =  \dfrac{\pi}{2}\norm{f}_\infty\norm{A}_2^2.
\end{align*}
Therefore by Hahn-Banach theorem we conclude that $$\norm{\eta_o}_{L^1([0,2\pi])}=\sup\limits_{f\in L^\infty([0,2\pi]):\norm{f}_\infty=1}\Bigg|\int_{0}^{2\pi} f(t) \eta_o(t) dt\Bigg|\leq  \dfrac{\pi}{2} \norm{A}_2^2.$$ 
This completes the proof. 
\end{proof}

	\section{Reduction to the finite dimension}
	The following lemma deals with the fact that given a unitary operator $U_0$, by suitably rotating the spectrum of $U_0$, or equivalently defining a new unitary operator $U_0^{'}=e^{-i\phi}U_0$ we get a self-adjoint operator $H_0$ such that $U_0^{'}$ is the Cayley transform of $H_0$, that is $U_0^{'}=(i-H_0)(i+H_0)^{-1}$. Note that the proof of this lemma is available in \cite[Theorem 1.1]{MoSi96} but for reader's convenience we are providing a proof herewith.
	\begin{lma}
	 Let $U_0$ be an unitary operator in a separable Hilbert space $\mathcal{H}$. Then there exists $\phi\in (-\pi, \pi]$ such that $\left(e^{i\phi}+U_0\right)$ is one to one, and hence invertible. Furthermore, the operator
	 \begin{align}\label{Redeq1}
	  \nonumber H_0=-i\left(-e^{i\phi}+U_0\right)\left(e^{i\phi}+U_0\right)^{-1} & = 
	  i\left(I-e^{-i\phi}U_0\right)\left(I+e^{-i\phi}U_0\right)^{-1}\\
	  &\equiv i\left(I-U_0^{'}\right)\left(I+U_0^{'}\right)^{-1}
	 \end{align}
is self-adjoint. 
\end{lma}
\begin{proof}
 Since $\mathcal{H}$ is separable, then the eigenvalues of $U_0$ are at most countable. Therefore there exists some $\phi\in (-\pi, \pi]$ such that $-e^{i\phi}\notin \sigma_p(U_0)$ (set of eigenvalues of $U_0$) and hence $\left(I+U_0^{'}\right)$ is invertible, where $U_0^{'}= e^{-i\phi}U_0$. Note that the following identity 
 \begin{equation*}
  \text{Ran} \left(I+U_0^{'}\right)^{\perp}
  = \text{Ker} \left(I+U_0^{'*}\right) = \text{Ker} \left(I+U_0^{'}\right) =\{0\}
 \end{equation*}
implies that the operator $H_0$ in \eqref{Redeq1} is densely defined and furthermore $H_0$ is also symmetric in this domain. Next we also observe that the ranges of $i+H_0= 2i \left(I+U_0^{'}\right)^{-1}$ and of $i-H_0= 2iU_0^{'} \left(I+U_0^{'}\right)^{-1}$ are the whole Hilbert space since $\text{Ran} \Big\{\left(I+U_0^{'}\right)^{-1}\Big\} = \text{Dom} \left(I+U_0^{'}\right)=\mathcal{H}$ and $U_0^{'}$ is unitary. Thus $H_0$ is self-adjoint and hence the proof. 
\end{proof}
In this section we prove some estimates similar to those in Section 3 of \cite{ChSi,MoSi94,MoSi96} and use them
to reduce the problem in finite dimension.
Now we begin with a lemma  collecting some results \cite{ChSi,Kato, MoSi94,MoSi96} following from the Weyl-von Neumann type construction.

\begin{lma}\label{l1}
		Let $U_0$ and $H_0$ be as above. Then given a set of normalized vectors $\big\{f_{_l}\big\}_{1\leq l\leq L}$ in $\hil$ and $\epsilon>0$ there exist a finite rank projection $P$ such that
		\begin{align*}
		&(i)\hspace*{.1in} \norm{P^\perp f_{_l}}<\epsilon \quad \text{for} \quad 1\leq l\leq L,\\
		&(ii)\hspace*{.1in} P^\perp H_0P\in\hils~\text{and}~\norm{P^\perp H_0P}_2<\epsilon,\\
		&(iii)\hspace*{.1in}\norm{P^\perp (i\pm H_0)^{-1}P}_2<\epsilon,\\
		&(iv)\hspace*{.1in}\text{for~any~integer~}m,~\norm{P^\perp U_0^m P}_2<2|m|\epsilon.
		\end{align*}
	\end{lma}
	\begin{proof}
	Let $F(\cdot)$ be be the spectral measure associated with the self-adjoint operator $H_0$. As in the proof of Proposition 3.1 in \cite{ChSi} we set $a$, $F_k=F(\Delta_k),$ where $\Delta_k=\Big(\frac{2k-n-2}{n}a,\frac{2k-n}{n}a\Big]$ for $1\leq k\leq n$, and 
	\begin{align*}
		\gkl=\begin{cases}
		&\frac{F_k\fl}{\norm{F_k\fl}}\hspace*{.7in}\text{if}~ F_k\fl\neq 0,\\
		&0 \hspace*{1in}\text{if}~ F_k\fl= 0,
		\end{cases}
    \end{align*}
    for $1\leq k\leq n$ and $1\leq l\leq L$ in such a way so that $\left\|\big[I-F((-a,a])\big]f_l\right\|<\epsilon$ for $1\leq l\leq L$ and $\gkl\in F_k\mathcal{H}\subseteq ~\text{Dom} (H_0)$.
	Let $P$ be the orthogonal projection onto the subspace generated by 
	$\{g_{kl}:~~1\leq k\leq n;~~1\leq l\leq L\}$. We need to prove only $(iii)$ and $(iv)$ since the first two are given in Proposition 3.1 of \cite{ChSi}. Since $F_k$ commutes with $H_0$, $(H_0\pm i)^{-1}\gkl=F_k(H_0\pm i)^{-1}f_l/\|F_kf_l\|\in F_k\mathcal{H}$. Thus by setting $\lambda_k=\frac{2k-n-1}{n}a$ one has 
	\begin{align*}
		\norm{\big\{(H_0\pm i)^{-1}-(\lambda_k\pm i)^{-1}\big\}\gkl}^2&=\int_{\Delta_k}\Big|(\lambda\pm i)^{-1}-(\lambda_k\pm i)^{-1} \Big|^2~\norm{F(d\lambda)\gkl}^2\\
		&\leq\int_{\Delta_k}\Big|\lambda-\lambda_k \Big|^2~\norm{F(d\lambda)\gkl}^2 \leq \left(\frac{a}{n}\right)^2.
		\end{align*}
		It is clear that $P^{\perp}(H_0\pm i)^{-1}\gkl\in F_k\hil$ and therefore we have for any $u\in\hil$ (using the Gram-Schmidt orthonormal set made out of $\{\gkl\}$ which are also in Dom ($H_0$)) 
		\begin{align*}
		\norm{P^\perp (H_0\pm i)^{-1}Pu}^2&=\left\|P^\perp (H_0\pm i)^{-1}\sum_{k=1}^{n}\sum_{l=1}^{L}\la u,\gkl\ra \gkl\right\|^2 =\left\|\sum_{k=1}^{n}\sum_{l=1}^{L}\la u,\gkl\ra P^\perp (H_0\pm i)^{-1}\gkl\right\|^2\\
		&=\sum_{k=1}^{n}~\left\|\sum_{l=1}^{L}\la u,\gkl\ra P^\perp (H_0\pm i)^{-1}\gkl\right\|^2\\
		&=\sum_{k=1}^{n}~\left\|\sum_{l=1}^{L}\la u,\gkl\ra P^\perp \big((H_0\pm i)^{-1}-(\lambda_k\pm i)^{-1}\big)\gkl\right\|^2\\
		&\leq \sum_{k=1}^{n}~\left[\sum_{l=1}^{L}|\la u,\gkl\ra| \left\|P^\perp \big((H_0\pm i)^{-1}-(\lambda_k\pm i)^{-1}\big)\gkl\right\|\right]^2\\
		&\leq \left(\frac{a}{n}\right)^2L \sum_{k=1}^{n}~\left(\sum_{l=1}^{L}|\la u,\gkl\ra|\right)^2
		\leq \left(\frac{a}{n}\right)^2 L \norm{u}^2.
		\end{align*} 
		Thus, the Hilbert-Schmidt norm can be estimated to be 
		\begin{equation}\label{Redeq2}
		 \norm{
		 P^\perp (H_0\pm i)^{-1}P}_2\leq \sqrt{\dim(P)}\norm{P^\perp (H_0\pm i)^{-1}P}\leq \sqrt{nL}~\Big(\frac{a}{n} \Big) \sqrt{L}= L\Big(\frac{a}{\sqrt{n}} \Big).
		\end{equation}
		Moreover for $m=\pm 1$ the following identity
		\begin{align*}
		P^\perp U_0^{\pm 1} P =P^\perp \Big[e^{\pm i\phi}(i\mp H_0)(i\pm H_0)^{-1}\Big]P
		& =P^\perp \Big[e^{\pm i\phi}\big\{2i(i\pm H_0)^{-1}-I\big\}\Big]P\\
		& =2i~ e^{\pm i\phi}P^\perp \Big[(i\pm H_0)^{-1}\Big]P
		\end{align*}
		along with the above equation \eqref{Redeq2} implies that $\left\|P^\perp U_0^{\pm 1} P\right\|_2\leq 2|\pm 1| L\Big(\frac{a}{\sqrt{n}} \Big)$
		and finally principle of mathematical induction procedure leads to $\left\|P^\perp U_0^{m} P\right\|_2\leq 2|m| L\Big(\frac{a}{\sqrt{n}} \Big)$ for general $m$. 
		 The proof concludes by choosing n sufficiently large.
	\end{proof}
\begin{lma}\label{l2}
	Let $U$ and $U_0$ be two unitary operators in a separable infinite dimensional Hilbert space $\mathcal{H}$ such that $U-U_0\in\hils$ and let $A$ be the corresponding self-adjoint operator in $\hils$ such that $U=e^{iA}U_0$. Then  given $\epsilon>0$, there exists a projection $P$  of finite rank such that for any integer $m$ and for all $t$ with $|t|\leq T$,
	\begin{align*}
		&(i)\hspace*{.1in}\norm{P^\perp U_0^mP}_2<2|m| \epsilon,~~\norm{P^\perp A}_2< 2\epsilon,\hspace*{8.5cm}\\
		&(ii)\hspace*{.1in}\norm{P^\perp e^{itA}P}_2<2T e^{T\norm{A}}~\epsilon,~~\norm{P^\perp U^mP}_2<2|m|(e^{\norm{A}}+1)~\epsilon.
	\end{align*}
\end{lma}
\begin{proof}
Let $ A(\cdot)=\sum\limits_{l=1}^{\infty}\tl\la \cdot ,\fl\ra\fl$ be the canonical form of $A$ with $\sum\limits_{l=1}^{\infty} \tl^2<\infty.$ Next choose $L$ in such a way so that $\left\|A-A_L\right\|_2= \sqrt{\sum\limits_{l=L+1}^{\infty} \tl^2}<\epsilon$, where $A_L(\cdot)=\sum\limits_{l=1}^{L}\tl\la \cdot ,\fl\ra\fl$ and $\epsilon'=\min \big\{\epsilon, \frac{\epsilon}{\sum\limits_{l=1}^{L}|\tl|}\big\}>0$. Next, we apply Lemma~\ref{l1} with $H_0$ as the correponding self-adjoint operator associated with $U_0$ (see \eqref{Redeq1}), $\{f_1,f_2,\ldots,f_L\}$ and $\epsilon'$ in place of $\epsilon$. Hence we get a finite rank projection $P$ in $\mathcal{H}$ such that 
\begin{align*}
 \norm{P^\perp f_{_l}}<\epsilon'<\epsilon \quad \text{for} \quad 1\leq l\leq L \quad \text{and}\quad \norm{P^\perp U_0^m P}_2<2|m|\epsilon'<2|m|\epsilon \quad \text{for any integer}~ m.
\end{align*}
 Furthermore, 
\begin{align*}
	\norm{P^\perp A}_2&\leq\norm{P^\perp \left(A-A_L\right)}_2+\norm{P^\perp A_L}_2\leq \norm{A-A_L}_2+\norm{P^\perp A_L}_2\\
	&< \epsilon+ \left\|\sum_{l=1}^{L}\tl\la \cdot ,\fl\ra P^{\perp}\fl\right\|_2
	< \epsilon+ \epsilon' \left(\sum_{l=1}^L|\tl|\right) < 2\epsilon.
\end{align*}
For $(ii)$, by the same calculation as in page 831 of \cite{MoSi94}, it follows that
\begin{align}\label{Redeq3}
\nonumber \alpha(t)= \norm{P^\perp e^{itA} P}_2 =& \norm{P^\perp (e^{itA}-I)P}_2\\
\leq&\norm{A}\int_{0}^{t}\alpha(s)~ds+T\norm{P^\perp AP}_2\leq\norm{A}\int_{0}^{t}\alpha(s)~ds+2T\epsilon \quad \text{for} \quad |t|\leq T
\end{align}
solving this Gronwall-type inequality \eqref{Redeq3} leads to 
\begin{align*}
 \alpha(t) = \norm{P^\perp e^{itA} P}_2
 \leq  2T\epsilon ~e^{t\norm{A}}\leq 
 2Te^{T\norm{A}}\epsilon \quad \text{uniformly for t with}\quad |t|\leq T.
\end{align*}
Moreover by using $(i)$ (for $m=\pm 1$) and $(ii)$ (for $t=\pm 1$) we conclude  
\begin{align*}
\norm{P^\perp UP}_2=\norm{P^\perp e^{iA}U_0P}_2=\norm{P^\perp e^{iA}(P^\perp+P)U_0P}_2 < 2(1+e^{\norm{A}})~\epsilon
\end{align*}
and 
\begin{align*}
 \norm{P^\perp U^{-1}P}_2=\norm{P^\perp U_0^{-1}e^{-iA}P}_2=\norm{P^\perp U_0^{-1}(P+P^{\perp})e^{-iA}P}_2 < 2(1+e^{\norm{A}})~\epsilon.
\end{align*}
Finally mathematical induction procedure leads to
$\norm{P^\perp U^mP}_2<2|m|(e^{\norm{A}}+1)~\epsilon$ for general $m$. This completes the proof. 
\end{proof}
\begin{lma}\label{l3}
	Let $U$ and $U_0$ be two unitary operators in a separable infinite dimensional Hilbert space $\mathcal{H}$ such that $U-U_0\in\hils$ and let $A$ be the corresponding self-adjoint operator in $\hils$ such that $U=e^{iA}U_0$. Then for $\epsilon>0$ there exists a finite rank projection $P$ such that for any integers $m, k$ and $|s|\leq T$
	\begin{align*}
		&(i)\hspace*{.1in}\norm{P^\perp \left(e^{iA}-I\right)}_2<2\epsilon,~~\norm{\left(e^{isA}-e^{isA_P}\right)P}_2<2T\epsilon,\\
		&\hspace*{.27in}\bnorm{P^\perp(e^{iA}-iA-I)}_1<2\norm{A}_2\norm{A}^{-2}~(e^{\norm{A}}-\norm{A}-1)\epsilon,\\
		&(ii)\hspace*{.1in}\norm{\left(U_0^m-U_{0,P}^m\right)P}_2< 2|m|\epsilon,~~\norm{P\left(U^m-U_{P}^m\right)P}_2<2|m|\epsilon \left\{(|m|-1)e^{\|A\|}+(|m|+1)\right\},\\
		&(iii)\left|\operatorname{Tr}\left\{PU_P^m\left(e^{iA}-e^{iA_P}\right)U_0^k\right\}\right|<4 \epsilon^2 e^{\|A\|},
	\end{align*}
	where in the above $U_{0,P} =e^{i\phi}(i-PH_0P)(i+PH_0P)^{-1},$ $U_P=e^{(iPAP)}U_{0,P}$ and $A_P=PAP$.
\end{lma}
\begin{rmrk}\label{Rem1}
	 Now observe that $P$ commutes with $(i\pm PH_0P)$, $(i\pm PH_0P)^{-1}$ and $PAP$ and hence $P$ commutes with $U_{0,P}$ and $U_{P}$. Thus $PU_{0,P}P$ and $PU_{P}P$ can be looked upon as unitary operators on the Hilbert space $P\hil$
\end{rmrk}
\noindent $\textbf{Proof of Lemma 3.4.:}$
Given $U_0$ and $A$ construct $H_0$ and $P$ as in Lemma~\ref{l2} respectively.
\vspace{0.1in}

\noindent $(i)$~First we note that 
\begin{align}
\nonumber &\norm{P^\perp (e^{iA}-I)}_2=\Bnorm{ \int_0^1 iP^\perp Ae^{isA}ds}_2\leq \norm{P^\perp A}_2<2\epsilon,
\end{align}
\begin{align}
\nonumber&\bnorm{(e^{isA}-e^{isA_P})P}_2=\Bnorm{\int_{0}^{1} e^{istA}is(A-A_P)Pe^{is(1-t)A_p}dt}_2\leq T\norm{P^\perp AP}_2<2T\epsilon,
\end{align}
and furthermore
\begin{align*}
\bnorm{P^\perp(e^{iA}-iA-I)}_1=&~\Bnorm{P^\perp A^2\big(\sum_{k=2}^{\infty}\frac{(iA)^{k-2}}{k!}\big)}_1\\
\leq&~ \norm{P^\perp A}_2\norm{A}_2\Bnorm{\sum_{k=2}^{\infty}\frac{(iA)^{k-2}}{k!}}\leq~ 2\norm{A}_2\norm{A}^{-2}~(e^{\norm{A}}-\norm{A}-1)\epsilon.
\end{align*}
$(ii)$ Now we set $U_{0}^{\#^m}=U_{0}^{\pm m}$ and $U_{0,P}^{\#^m}=U_{0,P}^{\pm m},~m\geq 1$. Thus by using Lemma~\ref{l1} $(ii)$, Remark~\ref{Rem1} and the identity 
\begin{equation*}
 \left(U_0^{\#}-U_{0,P}^{\#}\right)P=
 \mp 2ie^{\pm i\phi}(i\pm H_0)^{-1}\left[P^{\perp}H_0P\right](i\pm PH_0P)^{-1}P
\end{equation*}
we have 
\begin{align*}
&\Bnorm{(U_0^{\#^m}-U^{\#^m}_{0,P})P}_2=\Bnorm{\sum_{j=0}^{m-1}U_0^{\#^{m-j-1}}(U_0^{\#}-U_{0,P}^{\#})U^{\#^j}_{0,P}P}_2\\
&\leq 2\sum_{j=0}^{m-1}\Bnorm{U_0^{\#^{m-j-1}}(i\pm H_0)^{-1}\left[P^{\perp}H_0P\right](i\pm PH_0P)^{-1}U^{\#^j}_{0,P}P}_2\\
& \leq 2\sum_{j=0}^{m-1}\Bnorm{P^{\perp}H_0P}_2< 2|m|\epsilon.
\end{align*}
 
Now first we note that
\begin{align}\label{eqdiffesti}
\nonumber \norm{P(U-U_{P})P}_2=~&\bnorm{P(e^{iA}U_0-e^{PAP}U_{0,P})P}_2\\
\nonumber \leq~ &\bnorm{Pe^{iA}(U_0-U_{0,P})P}_2+\bnorm{P(e^{iA}-e^{iA_P})U_{0,P}P}_2\\
\leq~ &\bnorm{(U_0-U_{0,P})P}_2+\bnorm{P(e^{iA}-e^{iA_P})}_2 <~4\epsilon,
\end{align}
by using $(i)$, $(ii)$. Furthermore, since $P$ commutes with $U_P$, we have for $m\geq 1$
\begin{align*}
&\bnorm{P(U^{m}-U^{m}_{P})P}_2=\Bnorm{\sum_{j=0}^{m-1}PU^{{m-j-1}}(U-U_{P})U^{j}_{P}P}_2\\
&\leq \sum_{j=0}^{m-1}\left\{\Bnorm{PU^{^{m-j-1}}P^{\perp}(U-U_{P})U^{j}_{P}P}_2 + \Bnorm{PU^{^{m-j-1}}P(U-U_{P})PU^{j}_{P}P}_2\right\}\\
& \leq \sum_{j=0}^{m-1}\left\{2\Bnorm{PU^{^{m-j-1}}P^{\perp}}_2 + \Bnorm{P(U-U_{P})P}_2\right\}
<~2m\epsilon \left\{(m-1)e^{\|A\|}+(m+1)\right\},
\end{align*}
by using the above equation \ref{eqdiffesti} and Lemma~\ref{l2} $(ii)$. Finally the estimate for $m\leq -1$ follows by taking the adjoint.
\vspace{0.1in}

$(iii)$ Now by applying trace properties and using Lemma~\ref{l2} $(i)$, $(ii)$ we conclude that
\begin{align*}
& \Bigg|\operatorname{Tr}\Big\{PU_P^m\big(e^{iA}-e^{iA_P}\big)U_0^k\Big\}\Bigg|
= \Bigg|\operatorname{Tr}\Big[ PU_P^m \left(\int_{0}^{1}\left\{e^{isA}i(A-A_P)Pe^{i(1-s)A_P}\right\}ds\right)~U_0^k\Big]\Bigg|\\
& = \Bigg|\int_{0}^{1}\operatorname{Tr}\Big[PU_p^m e^{isA}P^\perp A Pe^{i(1-s)A_P}U_0^k\Big]~ds\Bigg|
=\Bigg|\int_{0}^{1}\operatorname{Tr}\Big[P^\perp A P~ e^{i(1-s)A_P}U_0^kPU_P^m ~Pe^{isA}P^\perp\Big]~ds\Bigg|\\
& \leq  \int_{0}^{1} \norm{P^\perp AP}_2\bnorm{Pe^{isA}P^\perp}_2~ds <~ 4 \epsilon^2 e^{\|A\|}.
\end{align*}

\begin{rmrk}\label{rm}
 We can reformulate the above set of lemmas by saying that there exists a sequence $\{P_n\}$ of finite rank projections such that for $m,k\in\mathbb{Z}$ and $|s|\leq T$,
 \begin{align*}
&(i)\hspace*{.1in}\norm{P_n^\perp H_0P_n}_2,\hspace*{.1in} \norm{P_n^\perp U_0^mP_n}_2,\hspace*{.1in} \norm{P_n^\perp U^mP_n}_2,\hspace*{.1in} \norm{P_n^\perp A}_2\hspace*{.1in}\longrightarrow 0~\text{as}~ n\longrightarrow \infty,\\
&(ii)\hspace*{.1in}\bnorm{P_n^\perp (e^{iA}-I)}_2,\hspace*{.1in} \bnorm{(U_0^m-U_{0,n}^m)P_n}_2, \hspace*{.1in} \bnorm{P_n(U^m-U_{n}^m)P_n}_2\hspace*{.1in}\longrightarrow 0~\text{as}~ n\longrightarrow\infty,\\
&(iii)\hspace*{.1in}\bnorm{(e^{isA}-e^{isA_{P_n}})P_n}_2,\hspace*{.1in} \bnorm{P_n^\perp e^{isA}P_n}_2,\hspace*{.1in} \left|\operatorname{Tr}\left\{P_nU_{n}^m\left(e^{iA}-e^{iA_{P_n}}\right)U_0^k\right\}\right|\hspace*{.1in}\longrightarrow 0~\text{as}~ n\longrightarrow \infty,\\
&(iv)\hspace*{.1in}\bnorm{P_n^\perp(e^{iA}-iA-I)}_1\longrightarrow0~\text{as}~ n\longrightarrow \infty,
\end{align*}
where $A_n=P_nAP_n$, $U_{0,n} =e^{i\phi}(i-P_nH_0P_n)(i+P_nH_0P_n)^{-1}$,  $U_n=e^{(iA_n)}U_{0,n}$ and $U_{s,n}=e^{(isA_n)}U_{0,n}$.
\end{rmrk}

The next theorem show how the above set of lemmas can be used to reduce the relevant problem into a finite dimensional one.

\begin{thm}\label{th3}
	Let $U$ and $U_0$ be two unitary operators in a separable Hilbert space $\hil$ such that $U-U_0\in\hils$ and let $A\in\hils$ be the corresponding self-adjoint operator as in Theorem \ref{th1} such that $U=e^{iA}U_0$. Let $U_s=e^{isA}U_0,~s\in\R$ and $p(\cdot)$ be any trigonometric polynomial on $\cir$ with complex coefficients. Then there exists a sequence  $\{P_n\}$ of finite rank projections in $\mathcal{H}$ such that
	\begin{align}\label{eqapp}
	\nonumber &\operatorname{Tr}\Big\{p(U)-p(U_0)-\at{\dds}{s=0}p(U_s) \Big\}\\
	&\hspace*{1in}=\lim_{n\to \infty}\operatorname{Tr}~\left[P_n\Big\{p(U_n)-p(U_{0,n})-\at{\dds}{s=0}p(U_{s,n}) \Big\}P_n\right],
\end{align}
where $A_n=P_nAP_n$, $U_{0,n} =e^{i\phi}(i-P_nH_0P_n)(i+P_nH_0P_n)^{-1}$,  $U_n=e^{(iA_n)}U_{0,n}$ and $U_{s,n}=e^{(isA_n)}U_{0,n}$.
\end{thm}
\begin{proof}
		It will be sufficient to prove the theorem for $p(\lambda)=\lambda^r,r\in\mathbb{Z}, \lambda\in\cir$. Note that for $r=0$, both sides of \eqref{eqapp} are identically zero. First we prove for $r\geq 1$. Using the sequence $\{P_n\}$ of finite rank projections as obtained in Lemma~\ref{l2} and Lemma~\ref{l3} and using an expression
        similar to \eqref{derivexpeq} in $\mathcal{B}(\mathcal{H})$, we have that
        
		\begin{align}\label{eq1}
		\nonumber&\operatorname{Tr}\Big\{\Big[p(U)-p(U_0)-\at{\dds}{s=0}p(U_s)\Big]- P_n\Big[p(U_n)-p(U_{0,n})-\at{\dds}{s=0}p(U_{s,n})\Big]P_n\Big\}\\\
		\nonumber&=\operatorname{Tr}\Bigg\{\Big[U^r-U_0^r- \sum_{j=0}^{r-1} U_{0}^{r-j-1}(iA)U_0^{j+1}\Big]- P_n\Big[U_n^r-U_{0,n}^r- \sum_{j=0}^{r-1} U_{0,n}^{r-j-1}(iA_n)U_{0,n}^{j+1}\Big]P_n\Bigg\}\\
		\nonumber&=\operatorname{Tr}\Bigg\{\Big[\sum_{j=0}^{r-1} U^{r-j-1}(U-U_0)U_0^{j} - \sum_{j=0}^{r-1} U_{0}^{r-j-1}(iA)U_0^{j+1}\Big]\\ \nonumber&\hspace*{3cm}-P_n\Big[\sum_{j=0}^{r-1} U_{n}^{r-j-1}P_n(U_n-U_{0,n})P_nU_{0,n}^j- \sum_{j=0}^{r-1} U_{0,n}^{r-j-1}(iA_n)U_{0,n}^{j+1}\Big]P_n\Bigg\}\\
		\nonumber&=\operatorname{Tr}\Bigg\{\Big[\sum_{j=0}^{r-1} U^{r-j-1}(e^{iA}-I)U_0^{j+1} - \sum_{j=0}^{r-1} U_{0}^{r-j-1}(iA)U_0^{j+1}\Big]\\ 
		\nonumber&\hspace*{3cm}-P_n\Big[\sum_{j=0}^{r-1} U_{n}^{r-j-1}P_n(e^{iA_n}-I)P_nU_{0,n}^{j+1}- \sum_{j=0}^{r-1} U_{0,n}^{r-j-1}(iA_n)U_{0,n}^{j+1}\Big]P_n\Bigg\}\\
		\nonumber&=\operatorname{Tr}\Bigg\{\sum_{j=0}^{r-1} \Big[ U^{r-j-1}(e^{iA}-iA-I)U_0^{j+1}+(U^{r-j-1}-U_{0}^{r-j-1})(iA)U_0^{j+1}\Big]\\ 
		\nonumber&\hspace*{1cm}-P_n\left(\sum_{j=0}^{r-1} \Big[ U_{n}^{r-j-1}P_n(e^{iA_n}-iA_n-I)P_nU_{0,n}^{j+1}+ ( U_{n}^{r-j-1}-U_{0,n}^{r-j-1})(iA_n)U_{0,n}^{j+1}\Big]\right)P_n\Bigg\}\\
		\nonumber&=\operatorname{Tr}\Bigg\{\sum_{j=0}^{r-1}\Big[ U^{r-j-1}(e^{iA}-iA-I)U_0^{j+1}- P_n U_{n}^{r-j-1}P_n(e^{iA_n}-iA_n-I)P_nU_{0,n}^{j+1}P_n\Big]\\ 
		&\hspace*{1cm}+\sum_{j=0}^{r-1}\Big[(U^{r-j-1}-U_{0}^{r-j-1})(iA)U_0^{j+1}-P_n( U_{n}^{r-j-1}-U_{0,n}^{r-j-1})P_n(iA_n)U_{0,n}^{j+1}P_n\Big]\Bigg\}.
		\end{align}
		Using the results obtained in Lemma~\ref{l2} and Lemma~\ref{l3}, the first term of the expression \eqref{eq1} leads to
		\begin{align}\label{eq2}
		\nonumber&\Bigg|\operatorname{Tr}\Bigg\{\sum_{j=0}^{r-1}\Big[ U^{r-j-1}(e^{iA}-iA-I)U_0^{j+1}-P_n U_{n}^{r-j-1}P_n(e^{iA_n}-iA_n-I)P_nU_{0,n}^{j+1}P_n\Big]\Bigg\}\Bigg|\\
		\nonumber=~&\Bigg|\operatorname{Tr}\Bigg\{\sum_{j=0}^{r-1}\Big[(U^{r-j-1}-U_{n}^{r-j-1})P_n(e^{iA}-iA-I)U_0^{j+1}+U^{r-j-1}P_n^\perp (e^{iA}-iA-I)U_0^{j+1}  \\
		\nonumber&\hspace*{3cm}+ U_{n}^{r-j-1}P_n(e^{iA}-iA-I-e^{iA_n}+iA_n+I)U_0^{j+1}\\
		\nonumber&\hspace*{4cm}+ U_{n}^{r-j-1}P_n(e^{iA_n}-iA_n-I)P_n(U_0^{j+1}-U_{0,n}^{j+1})\Big]\Bigg\}\Bigg|\\
		\nonumber=~&\Bigg|\operatorname{Tr}\Bigg\{\sum_{j=0}^{r-1}\Big[P_n(U^{r-j-1}-U_{n}^{r-j-1})P_n(e^{iA}-iA-I)U_0^{j+1}+P_n^\perp U^{r-j-1}P_n(e^{iA}-iA-I)U_0^{j+1}\\
		\nonumber&\hspace*{.5in}+U^{r-j-1}P_n^\perp (e^{iA}-iA-I)U_0^{j+1}+U_{n}^{r-j-1}P_n(e^{iA}-iA-e^{iA_n}+iA_n+)U_0^{j+1} \\
		\nonumber&\hspace*{4cm}+ U_{n}^{r-j-1}P_n(e^{iA_n}-iA_n-I)P_n(U_0^{j+1}-U_{0,n}^{j+1})\Big]\Bigg\}\Bigg|\\
		\nonumber \leq~&\sum_{j=0}^{r-1}\Big\{\bnorm{P_n(U^{r-j-1}-U_{n}^{r-j-1})P_n}_2~\bnorm{(e^{iA}-iA-I)}_2+\bnorm{P_n^\perp U^{r-j-1}P_n}_2~\bnorm{e^{iA}-iA-I}_2\\
		\nonumber&\hspace*{.5in}+~\bnorm{P_n^\perp (e^{iA}-iA-I)}_1~+\Big|\operatorname{Tr}\big(P_n U_{n}^{r-j-1}P_n(e^{iA}-iA-e^{iA_n}+iA_n)U_0^{j+1}P_n \big)\Big| \\
		\nonumber&\hspace*{4cm}+ \bnorm{(e^{iA_n}-iA_n-I)}_2\bnorm{P_n(U_0^{j+1}-U_{0,n}^{j+1})}_2\Big\}\\
		\nonumber\leq~&\big(e^{\norm{A}}-\norm{A}-1\big)\norm{A}^{-1}\sum_{j=0}^{r-1}\Big\{\bnorm{P_n(U^{r-j-1}-U_{n}^{r-j-1})P_n}_2+\bnorm{P_n^\perp U^{r-j-1}P_n}_2\Big\}\\
		\nonumber&+r\bnorm{P_n^\perp (e^{iA}-iA-I)}_1+\sum_{j=0}^{r-1}\Big|\operatorname{Tr}\big(P_n U_{n}^{r-j-1}P_n(e^{iA}-e^{iA_n})U_0^{j+1}P_n \big)\Big|\\
		\nonumber~&+\sum_{j=0}^{r-1}\bnorm{P_n U_{n}^{r-j-1}P_nAP_n^\perp U_0^{j+1}P_n}_1
		+ \big(e^{\norm{A}}-\norm{A}-1\big)\norm{A}^{-1}\sum_{j=0}^{r-1}\bnorm{P_n(U_0^{j+1}-U_{0,n}^{j+1})}_2\\
		\nonumber\leq~&\big(e^{\norm{A}}-\norm{A}-1\big)\norm{A}^{-1}\sum_{j=0}^{r-1}\Big\{\bnorm{P_n(U^{r-j-1}-U_{n}^{r-j-1})P_n}_2+\bnorm{P_n^\perp U^{r-j-1}P_n}_2\\
		\nonumber~&\hspace*{1.5in}+\bnorm{P_n(U_0^{j+1}-U_{0,n}^{j+1})}_2\Big\}+r\bnorm{P_n^\perp (e^{iA}-iA-I)}_1\\
		~&+\sum_{j=0}^{r-1}\Big|\operatorname{Tr}\big(P_n U_{n}^{r-j-1}P_n(e^{iA}-e^{iA_n})U_0^{j+1}P_n \big)\Big|+\bnorm{P_nAP_n^\perp}_2\sum_{j=0}^{r-1} \bnorm{P_n^\perp U_0^{j+1}P_n}_2,
		\end{align}
		and the estimate of the second term of the right hand side of (\ref{eq1}) is as follows
	\begin{align}\label{equuu}
	\nonumber&\Bigg|\operatorname{Tr}\Bigg(\sum_{j=0}^{r-1}\Big[(U^{r-j-1}-U_{0}^{r-j-1})AU_0^{j+1}-P_n( U_{n}^{r-j-1}-U_{0,n}^{r-j-1})P_nA_nU_{0,n}^{j+1}\Big]\Bigg)\Bigg|\\
	\nonumber~=&\Bigg|\operatorname{Tr}\Bigg(\sum_{j=0}^{r-1}\Big[\big\{\big(U^{r-j-1}-U_{0}^{r-j-1}\big)-\big( U_{n}^{r-j-1}-U_{0,n}^{r-j-1} \big)\big\}P_nAU_0^{j+1}\\
	\nonumber&+\big(U^{r-j-1}-U_{0}^{r-j-1}\big)P_n^\perp AU_0^{j+1}+\big( U_{n}^{r-j-1}-U_{0,n}^{r-j-1}\big)P_n(A-A_n)U_{0}^{j+1}\\
	\nonumber&+(U_{n}^{r-j-1}-U_{0,n}^{r-j-1})P_nA_nP_n(U_{0}^{j+1}-U_{0,n}^{j+1})\Big]\Bigg)\Bigg|\\
	\nonumber\leq \nonumber~&\sum_{j=0}^{r-1}\Bigg\{\Big(\bnorm{\big(U^{r-j-1}-U_{n}^{r-j-1}\big)P_n}_2+\bnorm{\big(U_{0}^{r-j-1}-U_{0,n}^{r-j-1}\big)P_n}_2 \Big)\bnorm{P_n AU_0^{j+1}}_2\\ \nonumber&+\bnorm{U^{r-j-1}-U_{0}^{r-j-1}}_2\bnorm{P_n^\perp AU_0^{j+1}}_2+\bnorm{(U_{n}^{r-j-1}-U_{0,n}^{r-j-1})P_n}_2\bnorm{P_n AP_n^\perp}_2\\
	& \nonumber \hspace{1.8in}+2\norm{A}_2\bnorm{P_n(U_{0}^{j+1}-U_{0,n}^{j+1})}_2\Bigg\}\\
	\nonumber\leq \nonumber~&\bnorm{ A}_2~\sum_{j=0}^{r-1}\Big\{\bnorm{\big(U^{r-j-1}-U_{n}^{r-j-1}\big)P_n}_2+\bnorm{\big(U_{0}^{r-j-1}-U_{0,n}^{r-j-1}\big)P_n}_2\\
	&+2 \bnorm{P_n(U_{0}^{j+1}-U_{0,n}^{j+1})}_2\Big\}+ \frac{r(r-1)}{2}\norm{A}_2\bnorm{P_n^\perp A}_2.
	\end{align}
		Now using all estimates listed in Remark \eqref{rm} we conclude that  the right hand sides of \eqref{eq2} and \eqref{equuu} tend to zero as $n$ approaches to infinity. Hence from (\ref{eq1}) we deduce the desire approximation (\ref{eqapp}). On the other hand for  $p(\lambda)=\lambda^{r},~r\leq -1$, we have 
	\begin{align}\label{12}
	\nonumber&\operatorname{Tr}\Big\{\Big[p(U)-p(U_0)-\at{\dds}{s=0}p(U_s)\Big]- P_n\Big[p(U_n)-p(U_{0,n})-\at{\dds}{s=0}p(U_{s,n})\Big]P_n\Big\}\\\
	\nonumber&=\operatorname{Tr}\Bigg\{\sum_{j=0}^{|r|-1}\Big( {U^*}^{|r|-j-1}U_0^*(e^{-iA}-1){U_0^*}^{j}+{U_{0}^*}^{|r|-j-1}U_0^*(iA){U_0^*}^{j}\Big)\\
	\nonumber&\hspace*{1in}- P_n\sum_{j=0}^{|r|-1}\Big({U_n^*}^{|r|-j-1}{U^*_{0,n}}(e^{-iA_n}-1){U^*_{0,n}}^{j} + {U_{0,n}^*}^{|r|-j-1}{U_{0,n}^*}(iA_n){U_{0,n}^*}^{j}\Big)P_n\Bigg\}\\
	\nonumber&=\operatorname{Tr}\Bigg\{\sum_{j=0}^{|r|-1}\Big[ {U^*}^{|r|-j-1}U_0^*(e^{-iA}+iA-I){U_0^*}^{j}- P_n {U_{n}^*}^{|r|-j-1}{U_{0,n}^*}P_n(e^{-iA_n}+iA_n-I)P_n{U_{0,n}^*}^{j}P_n\Big]\\ 
	&~-\sum_{j=0}^{|r|-1}\Big[({U^*}^{|r|-j-1}-{U_{0}^*}^{|r|-j-1}){U_0^*}(iA){U_0^*}^{j}-P_n( {U_{n}^*}^{|r|-j-1}-{U_{0,n}^*}^{|r|-j-1})U_{0,n}^*P_n(iA_n){U_{0,n}^*}^{j}P_n\Big]\Bigg\}.
	\end{align} 
	Similarly as above with an appropriate rearrangement and using Remark \ref{rm}, one can show that the right-hand side of (\ref{12}) approaches to zero as $n$ tends to infinity. This completes the proof. 
	\end{proof}
	\section{Existence of sift function}
	In this section, we derive the trace formula corresponding to the pair $(U,U_0)$. The following theorem is one of the main result in this section. 
\begin{thm}\label{th4}
		Let $U$ and $U_0$ be two unitary operators in a separable Hilbert space $\hil$ such that $U-U_0\in\hils$ and let $A\in\hils$ be the corresponding self-adjoint operator as in Theorem \ref{th1} such that $U=e^{iA}U_0$. Denote $U_s=e^{isA}U_0,~s\in\R$. Then for any  trigonometric polynomial $p(\cdot)$ on $\cir$ with complex coefficients, $\left\{p(U)-p(U_0)-\at{\dfrac{d}{ds}p(U_s)}{s=0}\right\}\in\mathcal{B}_1(\mathcal{H})$ and there exists a $L^1([0,2\pi])$- function $\eta$ (unique upto an additive constant) such that
		\begin{align*}
		\operatorname{Tr}\Big\{p(U)-p(U_0)-\at{\dfrac{d}{ds}p(U_s)}{s=0}\Big\}=\int_{0}^{2\pi} \dfrac{d^2}{dt^2} \big\{p (e^{it})\big\} \eta(t) dt.
		\end{align*}
		Moreover, $\norm{\eta}_{L^1([0,2\pi])}\leq\dfrac{\pi}{2}\norm{A}_2^2.$
	\end{thm}

\begin{proof}
By Theorems \ref{th2} and \ref{th3}, we have that
\begin{align*}
		&\operatorname{Tr}\Big\{p(U)-p(U_0)-\at{\dfrac{d}{ds}p(U_s)}{s=0}\Big\}\\
		&=\lim_{n\to \infty}\operatorname{Tr}~\left[P_n\Big\{p(U_n)-p(U_{0,n})-\at{\dds}{s=0}p(U_{s,n}) \Big\}P_n\right]\\
		&=\lim\limits_{n\to \infty}\int_{0}^{2\pi} \dfrac{d^2}{dt^2} \big\{p(e^{it})\big\} \eta_{_n}(t)dt =\lim\limits_{n\to \infty}\int_{0}^{2\pi} \dfrac{d^2}{dt^2} \big\{p(e^{it})\big\}\eta_{o,n}(t)dt,
		\end{align*}
		where  
		\begin{equation}\label{boundeq}
       \eta_{o,n}(t) = \eta_n(t) -\frac{1}{2\pi} \int_0^{2\pi} \eta_n(s) ds~,~~t\in[0,2\pi]\quad \text{and}\quad \|\eta_{o,n}\|_{L^1([0,2\pi])}\leq \frac{\pi}{2} \|A\|_2^2. 
      \end{equation}
Next we want to show that $\big\{\eta_{o,n}\big\}$ is a Cauchy sequence in $L^1([0,2\pi])$. Indeed, for any $f\in L^\infty ([0,2\pi])$ we consider
\begin{align*}
f_o(t)=f(t)-\dfrac{1}{2\pi}\int_{0}^{2\pi}f(s)ds.
\end{align*} 
Now it is easy to observe that 
\begin{align*}
\int_{0}^{2\pi} f(t)\big\{\eta_{o,n}(t)-\eta_{o,m}(t)\big\}dt=\int_{0}^{2\pi} f_o(t)\big\{\eta_{n}(t)-\eta_{m}(t)\big\}dt,~ \int_{0}^{2\pi}f_o(t)dt=0 \quad \text{and}~\norm{f_o}_\infty\leq 2\norm{f}_\infty.
\end{align*} 
Therefore by following the idea contained in the paper of Gestezy et al.\cite{GePu} (see also \cite{ChSi}), using the expression \eqref{eqfin0} of $\eta$, using Fubini's theorem to interchange the orders of integration and integrating by-parts, we have for $g(e^{it})=\int\limits_{0}^{t} f_o(s)ds, ~t\in [0,2\pi]$ that
\begin{align*}
	&\int_{0}^{2\pi} f(t)\big\{\eta_{o,n}(t)-\eta_{o,m}(t)\big\}dt=\int_{0}^{2\pi} f_o(t)\big\{\eta_{_n}(t)-\eta_{_m}(t)\big\}dt \\
	&=\int_{0}^{2\pi}\ddt \big\{g(e^{it})\big\} \Bigg(\int_{0}^{1}\operatorname{Tr} \Big[A_n\big\{E_{0,n}(t)-E_{s,n}(t)\big\}-A_m\big\{E_{0,m}(t)-E_{s,m}(t)\big\}\Big]~ds\Bigg)~dt\\
	&=\int_{0}^{1} ds \int_{0}^{2\pi}\ddt \big\{g(e^{it})\big\}~\operatorname{Tr} \Big[A_n\big\{E_{0,n}(t)-E_{s,n}(t)\big\}-A_m\big\{E_{0,m}(t)-E_{s,m}(t)\big\}\Big]dt\\
	&=\int_{0}^{1} ds ~\Bigg(g(e^{it})~\operatorname{Tr} \Big[A_n\big\{E_{0,n}(t)-E_{s,n}(t)\big\}-A_m\big\{E_{0,m}(t)-E_{s,m}(t)\big\}\Big]\Bigg\lvert_{t=0}^{2\pi}\\
	&\hspace*{1in}-\int_{0}^{2\pi}g(e^{it})~\operatorname{Tr}
	\Big[A_n\big\{E_{0,n}(dt)-E_{s,n}(dt)\big\}-A_m\big\{E_{0,m}(dt)-E_{s,m}(dt)\big\}\Big]\Bigg)\\
	&=-\int_{0}^{1}ds\int_{0}^{2\pi}g(e^{it})~\operatorname{Tr} \Big[A_n\big\{E_{0,n}(dt)-E_{s,n}(dt)\big\}-A_m\big\{E_{0,m}(dt)-E_{s,m}(dt)\big\}\Big]\\
	&=\int_{0}^{1}ds~\operatorname{Tr} \Big[A_n\big\{g(U_{s,n})-g(U_{0,n})\big\}-A_m\big\{g(U_{s,m})-g(U_{0,m})\big\}\Big]\\
	&=\int_{0}^{1}ds~\operatorname{Tr} \Bigg[A_n\Big\{\big\{g(U_{s,n})-g(U_{s})\big\}-\big\{g(U_{0,n})-g(U_{0})\big\}\Big\}\\
	&\hspace*{1in}-A_m\Big\{\big\{g(U_{s,m})-g(U_{s})\big\}-\big\{g(U_{0,m})-g(U_{0})\big\}\Big\}+(A_n-A_m)\big\{g(U_{s})-g(U_{0})\big\}\Bigg],
	\end{align*}	
where $E_{s,n}(\cdot)$ and $E_{0,n}(\cdot)$ are the spectral measures determined uniquely by the unitary operators $U_{s,n}$ and $U_{0,n}$ respectively such that they are continuous at $t=0$ and noted that all the boundary terms vanishes. Next we note that as in \eqref{eqfin3}
\begin{align*}
 \nonumber P_n\big\{g(U_{s,n})-g(U_{s})\big\}P_n  = P_n\Bigg\{ \ipi\ipi\dfrac{g(e^{i\lambda})-g(e^{i\mu})}{e^{i\lambda}-e^{i\mu}} ~\mathcal{G}_n(d\lambda\times d\mu)\Big(P_n\big\{U_{s,n}-U_{s}\big\}\Big) \Bigg\}P_n,
\end{align*}
where $\mathcal{G}_n(\Delta\times \delta )(V)= E_{s,n}(\Delta)VE_s(\delta)$ ($V\in \mathcal{B}_2(\hil)$,~ $\Delta\times \delta \subseteq \mathbb{R}\times \mathbb{R}$ and $E_{s}(\cdot)$ is the spectral measure determined uniquely by the unitary operator $U_{s}$ such that it is continuous at $0$) extends to a spectral measure on $\mathbb{R}^2$ in the Hilbert space $\mathcal{B}_2(\hil)$ (equipped with the inner product derived from the trace) and its total variation is less than or equal to
$\|V\|_2$. Therefore 
\begin{equation*}
 \left\|P_n\big\{g(U_{s,n})-g(U_{s})\big\}P_n\right\|_2 \leq \pi \norm{f}_\infty \norm{P_n\big\{U_{s,n}-U_{s}\}}_2,
\end{equation*}
since $\vline\dfrac{g(e^{i\lambda})-g(e^{i\mu})}{e^{i\lambda}-e^{i\mu}}\vline \leq \dfrac{\pi}{2} \|f_o\|_\infty\leq \pi \|f\|_{\infty}$, for $\lambda,\mu \in [0,2\pi]$. But on the other hand 
\begin{align}
 \nonumber & \bnorm{P_n(U_{s,n}-U_{s})}_2 
\leq \bnorm{P_n(e^{isA_n}-e^{isA})U_{0,n}+P_ne^{isA}(U_{0,n}-U_{0})}_2\\
\nonumber&\leq \bnorm{P_n(e^{isA_n}-e^{isA})}_2+\bnorm{P_ne^{isA}P_n(U_{0,n}-U_{0})}_2+\bnorm{P_ne^{isA}P_n^\perp(U_{0,n}-U_{0})}_2\\
\nonumber&\leq \bnorm{P_n(e^{isA_n}-e^{isA})}_2+\bnorm{P_n(U_{0,n}-U_{0})}_2+2\bnorm{P_ne^{isA}P_n^\perp}_2\\
\nonumber &\leq \bnorm{P_nAP_n^\perp}_2+\bnorm{P_n(U_{0,n}-U_{0})}_2+2 s \bnorm{AP_n^\perp}_2
\end{align}
and hence 
\begin{equation}\label{mainest2}
 \Big|\operatorname{Tr} \Big[A_n\big\{g(U_{s,n})-g(U_{s})\big\}\Big]\Big|
 \leq \pi \norm{f}_\infty \norm{A}_2 \Big\{\bnorm{P_nAP_n^\perp}_2+\bnorm{P_n(U_{0,n}-U_{0})}_2+2 s \bnorm{AP_n^\perp}_2\Big\}.
\end{equation}
Similarly we conclude that
\begin{equation}\label{mainest3}
 \Big|\operatorname{Tr} \Big[A_n\big\{g(U_{0,n})-g(U_{0})\big\}\Big]\Big|
 \leq \pi \norm{f}_\infty \norm{A}_2 ~\norm{P_n(U_{0,n}-U_{0})}_2.
\end{equation}
Furthermore we also have
\begin{align}\label{mainest4}
\Big|\operatorname{Tr}\Big[(A_n-A_m)\big\{g(U_{s})-g(U_{0})\big\}\Big]\Big| & \nonumber \leq \pi \norm{f}_\infty~\norm{A_n-A_m}_2~\norm{U_s-U_0}_2\\
& \leq \pi \norm{f}_\infty~\norm{A_n-A_m}_2~(s\|A\|_2),
\end{align}
by using the estimate as in \eqref{mainest1}. Therefore using equations \eqref{mainest2},\eqref{mainest3} and \eqref{mainest4} 
we get 
\begin{align}
&\nonumber \Bigg|\int_{0}^{2\pi} f(t)\big\{\eta_{o,n}(t)-\eta_{o,m}(t)\big\}dt\Bigg|\\
& \leq  \nonumber \nonumber \int_{0}^{1}ds~\Bigg|\operatorname{Tr} \Bigg[A_n\Big\{\big\{g(U_{s,n})-g(U_{s})\big\}-\big\{g(U_{0,n})-g(U_{0})\big\}\Big\}\\
	& \nonumber \hspace*{1in}-A_m\Big\{\big\{g(U_{s,m})-g(U_{s})\big\}-\big\{g(U_{0,m})-g(U_{0})\big\}\Big\}+(A_n-A_m)\big\{g(U_{s})-g(U_{0})\big\}\Bigg]\Bigg|\\
	& \nonumber \leq K_{m,n}\|f\|_{\infty},
\end{align}
where 
\begin{align*}
K_{m,n}= &\pi \norm{A}_2 \Bigg[~\Big\{\bnorm{P_nAP_n^\perp}_2+\bnorm{P_n(U_{0,n}-U_{0})}_2+ \bnorm{AP_n^\perp}_2 +\norm{P_n(U_{0,n}-U_{0})}_2\Big\}\\
& \hspace{1in} +  ~\Big\{\bnorm{P_mAP_m^\perp}_2+\bnorm{P_m(U_{0,m}-U_{0})}_2+ \bnorm{AP_m^\perp}_2 +\norm{P_m(U_{0,m}-U_{0})}_2\Big\}\\
& \hspace{3in} + \frac{1}{2}~\|A_n-A_m\|_2\Bigg].
\end{align*}
Therefore by Hahn-Banach theorem $$\norm{\eta_{_{o,n}}-\eta_{_{o,m}}}_1=\sup\limits_{f\in L^\infty([0,2\pi]):\norm{f}_\infty=1}\Bigg|\int_{0}^{2\pi} f(t)\big\{\eta_{o,n}(t)-\eta_{o,m}(t)\big\}dt\Bigg|\leq K_{m,n}\to0~\text{as~}m,n\to\infty,$$
by using Remark \ref{rm} and hence $\{\eta_{_{o,n}}\}$ is a Cauchy sequence in $L^1([0,2\pi])$. Therefore there exists a $\eta\in L^1([0,2\pi])$ such that $\eta_{_{o,n}}$ converges to $\eta$ in
$L^1([0,2\pi])$ norm. Thus 
\begin{align}\label{traceformulaeq}
		\operatorname{Tr}\Big\{p(U)-p(U_0)-\at{\dfrac{d}{ds}p(U_s)}{s=0}\Big\}=\lim_{n\rightarrow \infty}\int_{0}^{2\pi} \dfrac{d^2}{dt^2} \big\{p (e^{it})\big\} \eta_{o,n}(t) dt
		=\int_{0}^{2\pi} \dfrac{d^2}{dt^2} \big\{p (e^{it})\big\} \eta(t) dt.
\end{align}
Moreover, from \eqref{boundeq} it follows that 
$\norm{\eta}_{L^1([0,2\pi])}\leq\dfrac{\pi}{2}\norm{A}_2^2.$ 
Regarding uniqueness of $\eta$, let $\eta_1$ and $\eta_2$ be two $L^1([0,2\pi])$ functions which satisfy \eqref{traceformulaeq} for any polynomial $p(\cdot)$ on $\mathbb{T}$. Now by considering $p(z)=z^n$ for $n\in \mathbb{Z}\setminus \{0\}$ we get
\[
 \int_{0}^{2\pi} e^{int} ~\big\{\eta_1(t)-\eta_2(t)\big\} dt =0\quad \quad \forall n \in \mathbb{Z}\setminus \{0\},
\]
and consequently uniqueness of Fourier series implies $(\eta_1-\eta_2)$ is constant. This completes the proof. 
\end{proof}
Our next aim is to extend the class of functions $\phi$ for which the trace formula \eqref{intequ4} hold true. 
\begin{lma}\label{l4}
	Let $f_n(s)=a_nU^n_s,$ where $a_n\in \mathbb{C}$ and $U_s=e^{isA}U_0$ as in the statement of Theorem~\ref{th4} be such that $\sum\limits_{n=-\infty}^{\infty}n^2|a_n|<\infty$. Then 
	\begin{equation}\label{eqextendclass}
	\at{\dds}{s=0}\left(\sum_{n=-\infty}^{\infty} f_n(s)\right)=\sum_{n=-\infty}^{\infty}~\left(\at{\dds}{s=0} f_n(s)\right),
	\end{equation} 
	where the infinite series on both sides of \eqref{eqextendclass} converge in operator norm.
\end{lma}
\begin{proof}  
The expression in \eqref{derivexpeq} along with the fact 
$\sum\limits_{n=0}^{\infty}n^2|a_n|<\infty$  implies both infinite series in \eqref{eqextendclass} converge in operator norm. Next we denote
 $\tau_n=\operatorname{sgn}(n), n\in\mathbb{Z}$. Then the definition of G$\hat{a}$teaux derivative and the following estimate  
{\begin{align*}
&\left\|\frac{1}{s}\left[\sum_{n=-\infty}^{\infty}a_n{U_s^{\tau_n}}^{|n|}-\sum_{n=-\infty}^{\infty}a_n{U_0^{\tau_n}}^{|n|}\right]-\sum_{n=-\infty}^{\infty}a_n 
\begin{cases}
\sum\limits_{j=0}^{|n|-1}U_0^{|n|-j-1}~(iA)~ U_s^{j+1} &\text{ if } n\geq 1,\\
\qquad 0 &\text{ if } n=0,\\
-\sum\limits_{j=0}^{|n|-1}(U_0^*)^{|n|-j}~(iA)~( U_0^*)^{j} &\text{ if } n\leq -1,
\end{cases}\right\|\\
\leq &\left\{\sum_{n=1}^{\infty}\left(\Big\{|a_n|+|a_{-n}|\Big\}\cdot\left[\frac{n(n-1)}{2}\norm{A}^2+n\left(e^{\norm{A}}-\norm{A}-1\right)\right]\right)\right\}\cdot|s|\quad\longrightarrow 0 \text{ as } s\longrightarrow  0,
\end{align*}}
yields equation \eqref{eqextendclass}. 
\end{proof}
Let $\mathcal{A}_\cir:=\Big\{\Phi\big\lvert~\Phi:\cir\to\mathbb{C},~ \Phi(z)=\sum\limits_{n=-\infty}^{\infty}a_nz^n ~\text{with}~\sum\limits_{n=-\infty}^{\infty}n^2|a_n|<\infty\Big\}$.
\begin{thm}\label{th5}
	Let $U$ and $U_0$ be two unitary operators in an  infinite dimensional separable Hilbert space $\hil$ such that $U-U_0\in\hils$. Then for any $\Phi\in\mathcal{A}_\cir$,  $\left\{\Phi(U)-\Phi(U_0)-\at{\dfrac{d}{ds}\Phi(U_s)}{s=0}\right\}\in\boh$ and there exists a $L^1([0,2\pi])$- function $\eta$, unique up to an additive constant, such that
	\begin{align*}
	\operatorname{Tr}\Big\{\Phi(U)-\Phi(U_0)-\at{\dfrac{d}{ds}}{s=0}\Phi(U_s)\Big\}=\int_{0}^{2\pi} \dfrac{d^2}{dt^2} \big\{\Phi (e^{it})\big\}\eta(t) dt.
	\end{align*}
\end{thm}
\begin{proof} Using the above Lemma \ref{l4} we have
	\begin{align}
	\nonumber\Phi(U)-\Phi(U_0)-\at{\dfrac{d}{ds}}{s=0}\Phi(U_s)=&\sum_{n=-\infty}^{\infty}a_n{U^{\tau_n}}^{|n|}-\sum_{n=-\infty}^{\infty}a_n{U_0^{\tau_n}}^{|n|}-\at{\dds}{s=0}\left(\sum_{n=-\infty}^{\infty}a_n{U_s^{\tau_n}}^{|n|}\right)\\
	\label{eq15}=&\sum_{n=-\infty}^{\infty}a_n\Bigg[{U^{\tau_n}}^{|n|}-{U_0^{\tau_n}}^{|n|}-\at{\dds}{s=0}{U_s^{\tau_n}}^{|n|}\Bigg].
	\end{align}
Moreover, using \eqref{derivexpeq} we conclude that $\Big({U^{\tau_n}}^{|n|}-{U_0^{\tau_n}}^{|n|}-\at{\dds}{s=0}{U_s^{\tau_n}}^{|n|}\Big)$  is trace class and the following trace norm estimate
\vspace{1in}

{\begin{align*}
	&\Bnorm{{U^{\tau_n}}^{|n|}-{U_0^{\tau_n}}^{|n|}-\at{\dds}{s=0}{U_s^{\tau_n}}^{|n|}}_1\\
	=&~\left\|{U^{\tau_n}}^{|n|}-{U_0^{\tau_n}}^{|n|}- \begin{cases}
		\sum\limits_{j=0}^{|n|-1}U_0^{|n|-j-1}~(iA)~ U_s^{j+1} &\text{ if } n\geq 1,\\
		\qquad 0 &\text{ if } n=0,\\
		-\sum\limits_{j=0}^{|n|-1}(U_0^*)^{|n|-j}~(iA)~( U_0^*)^{j} &\text{ if } n\leq -1
		\end{cases}\right\|_1\\
\leq~ &\left[\frac{|n|(|n|-1)}{2}+|n|\norm{A}^{-2}\left(e^{\norm{A}}-\norm{A}-1 \right)\right]\norm{A}_2^2
	\end{align*}}
	implies \begin{align*}
	\nonumber&\sum_{n=-\infty}^{\infty}|a_n|\Bnorm{U^n-U_0^n-\at{\dds}{s=0}U_s^n}_1\\
	\leq&~\sum_{n=1}^{\infty}(|a_n|+|a_{-n}|)\left[\frac{n(n-1)}{2}+n \norm{A}^{-2}\left(e^{\norm{A}}-\norm{A}-1 \right)\right]\norm{A}_2^2<\infty.
	\end{align*}
Therefore the series in \eqref{eq15} converges in trace norm and hence $\left\{\Phi(U)-\Phi(U_0)-\at{\dfrac{d}{ds}}{s=0}\Phi(U_s)\right\}$ is trace class and furthermore
	\begin{align}\label{eq16}
	\operatorname{Tr}\Big\{\Phi(U)-\Phi(U_0)-\at{\dfrac{d}{ds}}{s=0}\Phi(U_s)\Big\}=\sum_{n=-\infty}^{\infty}a_n~\operatorname{Tr}\Big[U^n-U_0^n-\at{\dds}{s=0}U_s^n\Big].
	\end{align}
	Thus by combining Theorem \ref{th4} and \eqref{eq16} and applying Fubinni's theorem we get
	\begin{align*}
	\operatorname{Tr}\Big\{\Phi(U)-\Phi(U_0)-\at{\dfrac{d}{ds}}{s=0}\Phi(U_s)\Big\}=&\sum_{n=-\infty}^{\infty}\ipi (-n^2a_ne^{int}) \eta(t)dt
	=\ipi\frac{d^2}{dt^2}~\big\{\Phi(e^{it})\big\} \eta(t) dt.
	\end{align*}
	This completes the proof. 
\end{proof}
\begin{crlre}
If $U$ and $U_0$ are two unitary operators in an infinite dimensional separable Hilbert space $\hil$ such that $U-U_0\in\hils$. Then there exists a $L^1([0,2\pi])$- function $\eta$, unique up to an additive constant, such that for any $z\in\mathbb{C}$ with $|z|\neq 1$,
\begin{align*}
\operatorname{Tr}\Big\{(U-z)^{-1}-(U_0-z)^{-1}-\at{\dds}{s=0}(U_s-z)^{-1} \Big\}=\ipi \frac{d^2}{dt^2}\big\{(e^{it}-z)^{-1}\big\} \eta(t)dt.
\end{align*}
\end{crlre}
    
\section*{Acknowledgements}
\textit{The research of the first named author is supported by the Mathematical Research Impact Centric
Support (MATRICS) grant, File No :MTR/2019/000640, by the Science and Engineering Research Board (SERB), Department of Science $\&$ Technology (DST), Government of India. The second and the third named author gratefully acknowledge the support provided by IIT Guwahati, Government of India.}

\end{document}